\documentclass[11pt]{amsart}
\usepackage{amsmath,amssymb,amsthm,array,color,multirow,pdflscape,graphicx,pigpen,stmaryrd,comment}
\usepackage[all]{xypic}
\usepackage[normalem]{ulem}

\usepackage{lmodern}

\usepackage{tikz}
\usepackage{tikz-cd}
\tikzset{shorten <>/.style={shorten >=#1,shorten <=#1}}

\usepackage{todonotes}

\usepackage[margin=1.5in]{geometry}

\usepackage[all,arc,2cell]{xy}
\UseAllTwocells

\usepackage{xcolor}
\definecolor{darkgreen}{rgb}{0,0.30,0} 
\definecolor{darkred}{rgb}{0.75,0,0}
\definecolor{darkblue}{rgb}{0,0,0.6} 
\usepackage[pdfborder=0,pagebackref,colorlinks,citecolor=darkgreen,linkcolor=darkgreen,urlcolor=darkblue]{hyperref}
\renewcommand*{\backref}[1]{}
\renewcommand*{\backrefalt}[4]{({%
    \ifcase #1 Not cited.%
          \or On p.~#2%
          \else On pp.~#2%
    \fi%
    })}

\def\makeautorefname#1#2{\expandafter\def\csname#1autorefname\endcsname{#2}}
\makeautorefname{equation}{Equation}%
\makeautorefname{footnote}{footnote}%
\makeautorefname{item}{item}%
\makeautorefname{figure}{Figure}%
\makeautorefname{table}{Table}%
\makeautorefname{part}{Part}%
\makeautorefname{appendix}{Appendix}%
\makeautorefname{chapter}{Chapter}%
\makeautorefname{section}{Section}%
\makeautorefname{subsection}{Section}%
\makeautorefname{subsubsection}{Section}%
\makeautorefname{paragraph}{Paragraph}%
\makeautorefname{subparagraph}{Paragraph}%
\makeautorefname{theorem}{Theorem}%
\makeautorefname{thm}{Theorem}%
\makeautorefname{addm}{Addendum}%
\makeautorefname{mainthm}{Main theorem}%
\makeautorefname{corollary}{Corollary}%
\makeautorefname{cor}{Corollary}%
\makeautorefname{lemma}{Lemma}%
\makeautorefname{lem}{Lemma}%
\makeautorefname{sublemma}{Sublemma}%
\makeautorefname{sublem}{Sublemma}%
\makeautorefname{subl}{Sublemma}%
\makeautorefname{prop}{Proposition}%
\makeautorefname{property}{Property}
\makeautorefname{pro}{Property}
\makeautorefname{sch}{Scholium}%
\makeautorefname{step}{Step}%
\makeautorefname{conject}{Conjecture}%
\makeautorefname{conj}{Conjecture}%
\makeautorefname{questn}{Question}
\makeautorefname{quest}{Question}
\makeautorefname{qn}{Question}
\makeautorefname{definition}{Definition}%
\makeautorefname{defn}{Definition}%
\makeautorefname{defi}{Definition}%
\makeautorefname{def}{Definition}%
\makeautorefname{dfn}{Definition}%
\makeautorefname{notation}{Notation}
\makeautorefname{notn}{Notation}
\makeautorefname{rem}{Remark}%
\makeautorefname{rems}{Remarks}%
\makeautorefname{rmk}{Remark}%
\makeautorefname{rk}{Remark}%
\makeautorefname{remarks}{Remarks}%
\makeautorefname{rems}{Remarks}%
\makeautorefname{rmks}{Remarks}%
\makeautorefname{rks}{Remarks}%
\makeautorefname{example}{Example}%
\makeautorefname{examp}{Example}%
\makeautorefname{exmp}{Example}%
\makeautorefname{exam}{Example}%
\makeautorefname{exa}{Example}%
\makeautorefname{axiom}{Axiom}%
\makeautorefname{axi}{Axiom}%
\makeautorefname{ax}{Axiom}%
\makeautorefname{case}{Case}%
\makeautorefname{claim}{Claim}%
\makeautorefname{clm}{Claim}%
\makeautorefname{assumpt}{Assumption}%
\makeautorefname{asses}{Assumptions}%
\makeautorefname{conclusion}{Conclusion}%
\makeautorefname{concl}{Conclusion}%
\makeautorefname{conc}{Conclusion}%
\makeautorefname{cond}{Condition}%
\makeautorefname{const}{Construction}%
\makeautorefname{con}{Construction}%
\makeautorefname{criterion}{Criterion}%
\makeautorefname{criter}{Criterion}%
\makeautorefname{crit}{Criterion}%
\makeautorefname{exercise}{Exercise}%
\makeautorefname{exer}{Exercise}%
\makeautorefname{exe}{Exercise}%
\makeautorefname{problem}{Problem}%
\makeautorefname{problm}{Problem}%
\makeautorefname{prob}{Problem}%
\makeautorefname{prob}{Problem}%
\makeautorefname{soln}{Solution}%
\makeautorefname{sol}{Solution}%
\makeautorefname{sum}{Summary}%
\makeautorefname{oper}{Operation}%
\makeautorefname{obs}{Observation}%
\makeautorefname{ob}{Observation}%
\makeautorefname{conv}{Convention}%
\makeautorefname{cvn}{Convention}%
\makeautorefname{warn}{Warning}%
\makeautorefname{note}{Note}%
\makeautorefname{fact}{Fact}%
\makeautorefname{ouch0}{Counterexample}%
\makeautorefname{introthm}{Theorem}%
\newtheorem{thm}{Theorem}[section]

\newtheorem{prop}{Proposition}[section]
\newtheorem{lem}{Lemma}[section]

\theoremstyle{definition}

\newtheorem{rem}{Remark}[section]

\newtheorem{introthm}{Theorem}

\makeatletter
\let\c@cor=\c@thm
\let\c@prop=\c@thm
\let\c@lem=\c@thm
\let\c@conj=\c@thm
\let\c@defn=\c@thm
\let\c@df=\c@thm
\let\c@exmp=\c@thm
\let\c@rem=\c@thm
\let\c@sch=\c@thm
\let\c@equation\c@thm
\makeatother

\newcommand{\Map}{\textup{Map}}

\newcommand{\id}{\textup{id}}

\newcommand{\hocolim}{\textup{hocolim}\,}
\newcommand{\holim}{\textup{holim}\,}
\newcommand{\ra}{\longrightarrow}

\newcommand{\simar}{\overset\sim\longrightarrow}

\newcommand{\mc}{\mathcal}
\newcommand{\op}{\textup{op}}

\newcommand{\sC}{\mathcal{C}}

\newcommand{\bA}{\textbf{A}}

\newcommand{\ob}{\textup{ob}\, }
\newcommand{\bG}{\mathcal{B}G_{\sbt}}
\newcommand{\cG}{C(\mathcal{B}G_{\sbt})}

\newcommand{\tG}{\mathcal{E}G}

\newcommand{\po}{\ar@{}[dr]|{\text{\pigpenfont R}}}
\newcommand{\pb}{\ar@{}[dr]|{\text{\pigpenfont J}}}

\DeclareMathOperator{\Cat}{\mathrm{Cat}}
\DeclareMathOperator{\Sing}{\mathrm{Sing}}
\DeclareMathOperator{\Tw}{\mathrm{Tw}}

\newcommand{{\sbt}}{\,\begin{picture}(-1,1)(0.5,-1)\circle*{1.8}\end{picture}\hspace{.05cm}}

\newlength{\storeparskip}
\title{Coassembly is a homotopy limit  map}
\author{Cary Malkiewich}
\address{Department of Mathematics, Binghamton University}
\email{malkiewich@math.binghamton.edu}
\author{Mona Merling}
\address{Department of Mathematics, The University of Pennsylvania}
\email{mmerling@math.upenn.edu}

\begin{document}

\begin{abstract}
We prove a claim by Williams that the coassembly map is a homotopy limit map. As an application, we show that the homotopy limit map for the coarse version of equivariant $A$-theory agrees with the coassembly map for bivariant $A$-theory that appears in the statement of the topological Riemann-Roch theorem.
\end{abstract}

\dedicatory{In memory of Bruce Williams}
\maketitle


\begingroup%
\setlength{\parskip}{\storeparskip}
\tableofcontents
\endgroup%

\section{Introduction}

In the celebrated paper \cite{dww}, Dwyer, Weiss and Williams give index-theoretic conditions that are necessary and sufficient for a perfect fibration $E \to B$ to be equivalent to a fiber bundle with fibers compact topological (resp. smooth) manifolds.  In \cite{bruce}, Williams defines a bivariant version of $A$-theory for fibrations, which is contravariant in one variable and therefore comes with a coassembly map. He then reinterprets the condition from \cite{dww} as the condition that a certain class in bivariant $A$-theory (the Euler characteristic), after applying the coassembly map, lifts either along the assembly map or the inclusion of stable homotopy into $A(X)$.

In this paper, we show that coassembly maps in general agree with homotopy limit maps, the latter being more amenable to computations. In particular, this shows that the target of Williams's coassembly can be interpreted  as a homotopy fixed point spectrum, which has an associated homotopy fixed point spectral sequence that computes its homotopy groups. Together with well-known formulas for the assembly map, e.g. in \cite[6.2]{coassembly}, this means we get combinatorial formulas for each of the maps used in the statement of the bivariant topological and smooth Riemann-Roch theorems from \cite{bruce}.  

In general, the homotopy limit map is defined for any topological group $G$ and $G$-space or $G$-spectrum $X$ as the map from fixed points to homotopy fixed points,
\begin{equation*} X^G\to X^{hG}. \end{equation*}
Atiyah proved that for $KU$ with $C_2$-action induced by complex conjugation the homotopy limit map is an equivalence. In general, this is not the case, and the homotopy limit problem, beautifully described in \cite{homotopylimit}, asks how close the homotopy limit map is to being an equivalence. Some of the classical examples of interest are Segal's conjecture where $X=\mathbb{S}_G$, the sphere spectrum for $G$ finite, the Atiyah-Segal completion theorem, where $X=KU_G$, equivariant topological $K$-theory for $G$ compact Lie, and Thomason's theorem, where $X=KE$,   the algebraic $K$-theory of a finite Galois extension with Galois group action. In all of these cases, the homotopy limit map is shown to become an equivalence after suitable completion or inversion of an element in the homotopy groups of the fixed point spectrum. More recent solutions of homotopy limit problems appear in \cite{hukrizormsby}, \cite{rondingsetc}, \cite{drewhomotopylimit}, which study the homotopy limit problem for $KGL$, the motivic spectrum representing  algebraic $K$-theory, with $C_2$-action.


On the other hand, the coassembly map considered in \cite{bruce} is defined for any reduced contravariant homotopy functor $F$, whose domain is the category of spaces over $BG$. It is a natural transformation $F \to F_{\%}$, one that universally approximates $F$ by a functor that sends homotopy pushouts to homotopy pullbacks. It is formally dual to the assembly map of \cite{weisswilliamsassembly,DavisLuck}, which by \cite{HP, DavisLuck} coincides with the assembly map of the Farrell-Jones conjecture \cite{farrell_jones}. A comprehensive recent survey on assembly maps is given in \cite{luckassembly}. The coassembly map is also a close analog of the linear approximation map of embedding calculus \cite{goodwillie_weiss1, goodwillie_weiss2}. Further applications of the coassembly map appear in \cite{cohen_umkehr,raptis_steimle,coassembly}.
%


Our first result is a precise correspondence between these two constructions. {We only consider topological groups $G$ that are the realization of a simplicial group $G_{\sbt}$, and we focus on the case where $F$ takes values in spectra, because the corresponding result for spaces is similar and a little easier. Without loss of generality, we assume that the homotopy functor $F$ is enriched in simplicial sets, so that $F(EG)$ carries a continuous left action by $G$, and $F(BG)$ maps to its fixed points. We may then make $F(EG)$ into a $G$-spectrum whose fixed points are $F(BG)$. An analogue of this result for the assembly map can be found in \cite[\S 5.2]{DavisLuck}.
\begin{introthm}\label{intro_first}(\autoref{coassemblymap})
	Let $G$ be a group that is the realization of a simplicial group $G_{\sbt}$. The coassembly map on the terminal object $F(BG) \to F_{\%}(BG)$ is equivalent to the homotopy limit map of this $G$-spectrum, $F(BG) \to F(EG)^{hG}$.
\end{introthm}
This is similar to a claim in \cite{bruce}, when $F$ is a contravariant form of algebraic $K$-theory and $G \simeq \Omega X$. Giving a precise proof amounts to showing that diagrams on a suitable category of contractible spaces over $BG$ correspond to $G$-objects, plus a little more structure. Our version of the argument uses parametrized spectra to form a bridge between the two settings.


Our second result applies \autoref{intro_first} to Williams's bivariant $A$-theory functor $\bA(E \to B)$ to fibrations of the form $EG\times_G X\to BG$ where $G$ is a finite group. This gives the homotopy limit map of the ``coarse'' equivariant  $A$-theory $G$-spectrum from \cite{CaryMona}, equivalently the $K$-theory of group actions from \cite{Gmonster2} applied to retractive spaces over $X$. 


\begin{introthm}\label{intro_second}(\autoref{thm:homotopy_fixed_equals_coassembly})
In the stable homotopy category, the homotopy limit map for $\bA_G^\textup{coarse}(X)$ is isomorphic to the coassembly map for bivariant $A$-theory:
\[ \xymatrix{
\bA_G^{\textup{coarse}}(X)^H \ar[d]^-\sim_-\Phi \ar[r] & \bA_G^{\textup{coarse}}(X)^{hH} \ar[d]^-\sim \\
\bA(EG \times_H X \ra BH) \ar[r] & \bA_{\%}(EG \times_H X \ra BH).} \]
\end{introthm}

This is not quite a direct consequence of \autoref{intro_first} because we have to show that the equivalence between the two theories preserves the $G$-actions and inclusions of fixed points, up to some coherent homotopies.

\begin{rem}
This provides one half of an argument that would significantly generalize the main theorem of \cite{coassembly}. The other half relies on a conjectural connection between assembly maps and the Adams isomorphism, which we do not pursue here.
\end{rem}

\begin{rem}
This paper does not consider the homotopy limit problems for profinite groups, which involve a modified definition of homotopy fixed points that are associated to the continuous cohomology of the profinite group, see \cite{devinatzhopkins}. Our homotopy limit map is the usual one from e.g. \cite[Chapter XI, 3.5.]{bousfield_kan}, and we only consider those topological groups that are geometric realizations of simplicial groups. The main example we have in mind is $\Omega X$. 
\end{rem}

\textbf{Conventions.} Throughout all of our topological spaces are compactly generated weak Hausdorff (CGWH), see \cite[Appendix A]{lewis_thesis} and \cite{strickland}. Unless otherwise noted, the term ``spectra'' can be interpreted to mean prespectra, symmetric spectra, or orthogonal spectra. See \cite{mmss} for more information about how to pass between these different models. The term ``na\"ive $G$-spectrum'' refers to a spectrum with an action by the group $G$, up to maps that are equivalences on all of the categorical fixed point spectra $X^H$ subgroups $H \leq G$. Equivalently, this can be viewed as a diagram of spectra on the orbit category $\mc O(G)^{op}$. In fact, we will only be concerned with diagrams restricted to the trivial orbit $G/G$ and the full orbit  $G/e$, corresponding to the data of the $G$-fixed points of a na\"ive $G$-spectrum and its underlying spectrum with $G$-action.

\subsection*{Acknowledgements}

We thank Mike Hill and the anonymous referee of \cite{CaryMona} for helpful feedback that contributed to this project. We are greatly indebted to the anonymous referee of this paper for a very careful reading that substantially improved the paper, and for suggesting the proof of \autoref{contractible}, which is much simpler and more elegant than the one that appeared in the first version. We would also like to thank George Raptis for all the insights he has shared with us about bivariant $A$-theory and the Dwyer-Weiss-Williams theorem during the Junior Trimester Program at the Hausdorff Institute in Bonn when we were all part of the ``New directions in $A$-theory" group, and we thank Jim Davis for illuminating discussions about his work with Wolfgang L{\"u}ck on the dual case of the assembly map. Lastly, we very much thank the Max Planck Institute in Bonn for their hospitality while much of this paper was written. The second named author also acknowledges support from NSF grant DMS 1709461/1850644.

\section{Review of coassembly}

Let $B$ be an unbased space and let $\mc U_B$ denote the comma category of spaces over $B$. A commuting square in $\mc U_B$ is a homotopy pushout square if it is such when we forget the maps to $B$. A contravariant functor $F$ from $\mc U_B$ to spectra is
\begin{itemize}
	\item \textbf{reduced} if it sends $\emptyset \to B$ to a weakly contractible spectrum,
	\item a \textbf{homotopy functor} if it sends weak equivalences of spaces to stable equivalences of spectra, and
	\item \textbf{excisive} if it is a reduced homotopy functor that sends coproducts and homotopy pushout squares of spaces to products and homotopy pullback squares of spectra, respectively.
\end{itemize} 

Note that this last condition can be stated in several equivalent ways, the simplest of which is that $F$ takes all homotopy colimits to homotopy limits. 

If $F$ is a contravariant reduced homotopy functor from $\mc U_B$ to spectra, consider the comma category of excisive functors $P$ with natural transformations $F \to P$. Define a weak equivalence of such functors to be a natural transformation $P \to P'$ (under $F$) that is a stable equivalence at every object. Inverting these equivalences gives the homotopy category of excisive functors under $F$.

\begin{prop}(see \cite[5.4]{cohen_umkehr},\cite[5.4]{coassembly},\cite[\S 7]{malkiewich2015tower})
	The homotopy category of excisive functors under $F$ has an initial object $F_{\%}$, in other words a universal approximation of $F$ by an excisive functor. The natural transformation $F \to F_{\%}$ can be given by the formula
	\[ F(X \to B) \to \underset{(\Delta^n \to X) \in \Delta_{X}^\op}\holim\, F((\Delta^n \amalg B) \to B). \]
\end{prop}

Here $\Delta_{X} = \Delta_{\Sing X}$ is the category of simplices in the simplicial set $\Sing X$.
Concretely, it has an object for every continuous map $\Delta^n \to X$ and a morphism for every factorization $\Delta^p \to \Delta^q \to X$ where $\Delta^p \to \Delta^q$ is a composite of inclusions of a face. There is a natural ``last vertex'' operation that gives an equivalence $|\Delta_X| \simar X$ \cite[III.4]{goerss_jardine},\cite[\S 5]{coassembly}.

We could alternatively describe $F_{\%}(X \to B)$ as the spectrum of sections of a parametrized spectrum over $X$ whose fiber over $x$ is $F((x \amalg B) \to B)$. See \cite{weiss_williams_assembly}, \cite{bruce}, \cite{cohen_umkehr}, \cite{malkiewich2015tower}, and \cite{coassembly} for more details and other explicit constructions of the coassembly map.

\section{Proof of  \autoref{intro_first}}\label{coassemblysection}

The first step is to interpret both the homotopy limit map and the coassembly map as the unit of an adjunction.

Let $G_{\sbt}$ be a simplicial group with realization $G = |G_{\sbt}|$, and let $BG$ be the topological bar construction of $G$. It will be convenient for us to let $\mc U_{BG}$ refer to the category of unbased spaces over $BG$ that are homotopy equivalent to cell complexes, as opposed to all spaces over $BG$. Recall that $\Delta_{BG} \subseteq \mc U_{BG}$ is the subcategory of spaces over $BG$ consisting only of the simplices $\Delta^p \to BG$ for varying $p \geq 0$ and the compositions of face maps. Note that a homotopy functor on this subcategory must send every map to a weak equivalence.

\begin{prop}\label{adjunction2}
For reduced homotopy functors on spaces over $BG$, the coassembly map is the unit of the adjunction of homotopy categories
\[ \xymatrix @C=1.5in @R=0.5in{
*+<16pt>[F-:<16pt>]\txt{ Reduced homotopy functors \\ $F: \mc U_{BG}^\op$ $ \rightarrow \mc Sp$ } \ar@/^/[r]^-*\txt{restrict} \ar@{}[r]|-{\perp} &
*+<16pt>[F-:<16pt>]\txt{ Homotopy functors \\ $F: \Delta_{BG}^\op$ $ \rightarrow \mc Sp$ } \ar@/^/[l]^-*\txt{$\underset{\Delta^p \to X}\holim F(\Delta^p)$}
} \] 
\end{prop}

\begin{proof}
We first examine the larger homotopy category of all functors. It is standard that the homotopy right Kan extension is the right adjoint of restriction. Furthermore, the canonical map of $F$ into the extension of the restriction of $F$ is the unit of this adjunction. By \cite[\S 5]{cohen_umkehr} or \cite[\S 7]{malkiewich2015tower}, this particular model for the homotopy right Kan extension sends homotopy functors to reduced homotopy functors, so the adjunction descends to these subcategories, with the same unit.
\end{proof}

Let $\bG$ be the simplicially enriched category with one object $[e]$ and morphism space $G_{\sbt}$. Note that $BG \cong |\bG|$. Let $\cG$ be the ``cone'' category with one additional object $[G]$ and one additional nontrivial morphism $[G] \to [e]$. This is isomorphic to the full subcategory of the enriched orbit category $\mc O(G)^\op$ on the orbits $G/e$ and $G/G$. Let $\iota\colon \bG \to \cG$ be the inclusion.

\begin{rem}\label{B'G} If $X$ is a $G$-space or na\" ive $G$-spectrum then $X^G$ and $X = X^{\{e\}}$ form a diagram over $\cG$. If $X$ is a genuine orthogonal $G$-spectrum, the same is true for the genuine fixed points $X^G$, by taking a fibrant replacement then passing to the underlying na\" ive $G$-spectrum.
\end{rem}

\begin{prop}\label{adjunction1}
	For na{\"i}ve $G$-spectra, the map $(-)^G \to (-)^{hG}$ is equivalent to the unit of the adjunction of homotopy categories
	\[ \resizebox{\textwidth}{!}{
		\xymatrix @C=1.5in @R=0.5in{
		*+<16pt>[F-:<16pt>]\txt{ Enriched $\cG$ diagrams of spectra } \ar@/^/[r]^-*\txt{$\iota^*$} \ar@{}[r]|-{\perp} &
		*+<16pt>[F-:<16pt>]\txt{ Enriched $\bG$ diagrams of spectra \\ (i.e. spectra with $G$-action) } \ar@/^/[l]^-*\txt{enriched homotopy \\ right Kan extension}
	}
	} \] 
	evaluated at $[G]$. 
\end{prop}
\begin{proof}
	This is immediate from the local formula for an enriched homotopy right Kan extension \cite[7.6.6]{riehl_cat_htpy}.
\end{proof}

The next step is to relate the categories on the left-hand side of these adjunctions together. Morally, we take each homotopy functor $F$ to the diagram on $\cG$ given by $F(BG)$ and $F(EG)$.

There are two problems to address here. The first problem is that this is not an equivalence of homotopy categories, but we can fix that by localizing the category of homotopy functors along the maps that are equivalences on $BG$ and $EG$. The second problem is that $G$ will not act on $F(EG)$ unless we make $F$ simplicially enriched. We fix the second problem using the following result.

\begin{lem}\label{replace_by_enriched}
	Every contravariant homotopy functor $F$ to spaces or spectra can be replaced by a simplicially enriched functor, by a zig-zag of equivalences of functors
	\[ \xymatrix{ F & \ar[l]_-\sim F' \ar[r]^-\sim & \widetilde{F}' } \]
	that is itself functorial in $F$.
\end{lem}
\begin{proof}
	This is by a variant of the trick used in \cite{waldhausen} to replace functors by homotopy functors. It adapts from covariant to contravariant functors by replacing $\Map(\Delta^p,-)$ with $\Delta^p \times -$. 
		
	If $F$ lands in orthogonal spectra, regard it as landing in prespectra or symmetric spectra, and replace the spectrum $F(X)$ at each level by by $F'(X) = |\Sing F(X)|$. The effect of this is that each degeneracy map $\Delta^p \to \Delta^q$ induces a levelwise cofibration $F'(\Delta^q \times X) \to F'(\Delta^p \times X)$. 
	Then pass back up to orthogonal spectra if desired, and replace $F'(X)$ again by the realization
	\[ \widetilde{F}'(X) = | n \mapsto F'(\Delta^n \times X) |. \] 
	This defines a functor that receives a map from $F'$ by inclusion of simplicial level 0. The map is an equivalence on each spectrum level, because $F'$ is a homotopy functor and the simplicial space defined above is good and therefore Reedy cofibrant \cite{Lillig}. We extend the functor structure on $\widetilde F'$ to a simplicial enrichment by taking each map $|Y_{\sbt}| \times X \to Z$ to the realization of the map that at level $k$ is
	\[ Y_k \times F'(\Delta^k \times X) \to F'(\Delta^k \times Z), \]
	obtained from the map of spaces
	\[ Y_k \times \Delta^k \times X \to \Delta^k \times Z \]
	whose coordinates are the action $Y_k \times \Delta^k \times X \to Z$ and the projection to $\Delta^k$.
\end{proof}

\begin{prop}
	The forgetful functors in the following diagram are equivalences of homotopy categories. Here ``enriched'' means simplicially enriched.
	\begin{figure}[h]\label{left_hand_equivalences}
		\[ 
			\xymatrix @C=1.5in @R=0.5in{
				*+<16pt>[F-:<16pt>]\txt{ Reduced homotopy functors \\ $F\colon \mc U_{BG}^\op \rightarrow \mc Sp$ (localized) }
				\\
				*+<16pt>[F-:<16pt>]\txt{ Enriched reduced homotopy functors \\ $F\colon \mc U_{BG}^\op \rightarrow \mc Sp$ (localized) }
				\ar[u]_-\sim
				\ar[d]^-\sim
				\\
				*+<16pt>[F-:<16pt>]\txt{ Enriched reduced functors \\ $F\colon \mc U_{BG}^\op \rightarrow \mc Sp$ (localized) }
				\ar[d]^-\sim
				\\
				*+<16pt>[F-:<16pt>]\txt{ Enriched functors \\ $\cG \to \mc Sp$ }
		} \] 
	\end{figure}
\end{prop}

\begin{proof}
	The construction of \autoref{replace_by_enriched} gives an inverse to the first equivalence. Note this is still well-defined after localizing because the construction preserves the property of a map of functors $F \to F'$ being an equivalence on one particular space $X$. For the second pair of categories, by Whitehead's theorem any enriched functor is a homotopy functor on the cofibrant and fibrant objects. Hence we can invert the forgetful functor by composing each $F$ with a fibrant replacement in $\mc U_{BG}$. To check this respects the localization, we note that when we turn an enriched functor into a homotopy functor, it will have equivalent values on $EG$ and $BG$, because these two spaces are already fibrant. For the final pair of categories, the restriction functor has the enriched homotopy right Kan extension as its right adjoint, and this adjunction clearly descends to the localization. In fact, since $C(\bG)$ is a full subcategory of $\mc U_{BG}^\op$, the counit is an equivalence, and therefore by the definition of our localization, the unit is also an equivalence, hence we get an equivalence of categories.
\end{proof}

Next we relate the categories on the right-hand side in \autoref{adjunction1} and \autoref{adjunction2} using parametrized spectra. To be definite, we will now assume that $\mc Sp$ means orthogonal spectra. The category of parametrized orthogonal spectra is defined in \cite[11.2.3]{ms}, and its homotopy category is obtained by inverting the $\pi_*$-isomorphisms from \cite[12.3.4]{ms}.

\begin{figure}[h]\label{right_hand_equivalences}
	\[ 
	\xymatrix @C=1.5in @R=0.5in{
		*+<16pt>[F-:<16pt>]\txt{ Homotopy functors \\ $F\colon \Delta_{BG}^\op \rightarrow \mc Sp$ }
		\ar[d]^-*\txt{$\underset{\Delta_{BG}^\op}\hocolim F(\Delta^p)$}_-\sim
		\\
		*+<16pt>[F-:<16pt>]\txt{ Parametrized spectra \\ over $|\Delta_{BG}^\op|$ }
		\\
		*+<16pt>[F-:<16pt>]\txt{ Enriched functors $\bG \to \mc Sp$ \\ (spectra with left $G$-action) }
		\ar[u]_-{l_!(QE \times_G -)}^-\sim
	} \] 
\end{figure}
	
The first part of the equivalence is as follows. Given a diagram $F$ of orthogonal spectra over $\mc C$, at each spectrum level we can take its Bousfield-Kan homotopy colimit as a diagram of unbased spaces, giving a retractive space over $|\mc C|$. In total this gives a parametrized spectrum $\underset{\mc C}\hocolim F$ over $|\mc C|$, see \cite[\S 4]{lind_malkiewich_morita}.

The second part of the equivalence is the Borel construction $EG \times_G -$, followed by pullback along the equivalence $|\Delta_{BG}^\op| \simar BG$. Alternatively, we make the following construction. Let $E$ be any weakly contractible space with a free right $G$-action, with a map $E/G \to |\Delta_{BG}^\op|$. Let $QE$ be its cofibrant replacement as a free $G$-space, so that there is an equivalence $l\colon QE/G \simar BG$. If $X$ is a spectrum with $G$-action, take a cofibrant replacement if necessary so that its levels are well-based, then take $QE \times_G X$, which is a parametrized spectrum over $QE/G$, and push it forward along $l$ to $|\Delta_{BG}^\op|$. We will see in the next proposition that this is always equivalent to the Borel construction, but it is convenient to allow ourselves to pick a particular space $E$ with this property, rather than having to use the pullback of $EG$ to $|\Delta_{BG}^\op|$.
%

\begin{prop}
	These are equivalences of homotopy categories, and the second is independent of the choice of $E$, up to isomorphism.
\end{prop}

\begin{proof}
	For the first one, the homotopy category of homotopy functors on $\Delta_{BG}^\op$ is equivalent to the homotopy category of functors that are fibrant in the aggregate model structure of \cite[Thm 4.4]{lind_malkiewich_morita}. Hence $\underset{\Delta_{BG}^\op}\hocolim F(\Delta^p)$ is naturally isomorphic as a map of homotopy categories to the left Quillen equivalence of \cite[Thm 4.5]{lind_malkiewich_morita}, and is therefore an equivalence. On the other hand, for a $G$-space $X$ the horizontal maps in the following square are equivalences:
	\[ \xymatrix{
		QE \times_G X \ar[d] \ar[r]^-\sim & EG \times_G X \ar[d] \\
		|\Delta_{BG}^\op| \ar[r]^-\sim & BG
	} \]
	Hence the functor $QE \times_G -$ is equivalent to the Borel construction $EG \times_G -$ (which lands in spectra over $BG$) followed by the pullback from $BG$ to $|\Delta_{BG}^\op|$. (Under the cofibrancy assumptions on $X$, the same is also true if we push $QE \times_G X$ forward along $l$.) This factorization into Borel-then-pullback also holds at the level of homotopy categories, since the Borel construction preserves all equivalences and outputs a fibration, on which the pullback preserves equivalences. Then the Borel construction is an equivalence by \cite[Appendix B]{ando_blumberg_gepner} or \cite[Thm 4.5]{lind_malkiewich_morita}, and the derived pullback is an equivalence by \cite[Prop 12.6.7]{ms}.
\end{proof}

Now we may finish the proof of  \autoref{intro_first}.

\begin{thm}\label{coassemblymap}
For any reduced homotopy functor $F\colon \mc U_{BG}^\op \to \mc Sp$, the coassembly map on $BG$ is isomorphic in the homotopy category to the map $F(BG) \to F(EG)^{hG}$ induced by the functoriality of $F$.
\end{thm}

\begin{proof}
The adjunction from \autoref{adjunction2} descends to the localization we described above, hence we get the following diagram of adjunctions and equivalences of homotopy categories. It remains to check that the equivalences and left adjoints in this figure commute up to some natural isomorphism, so that the figure is an ``equivalence of adjunctions.''

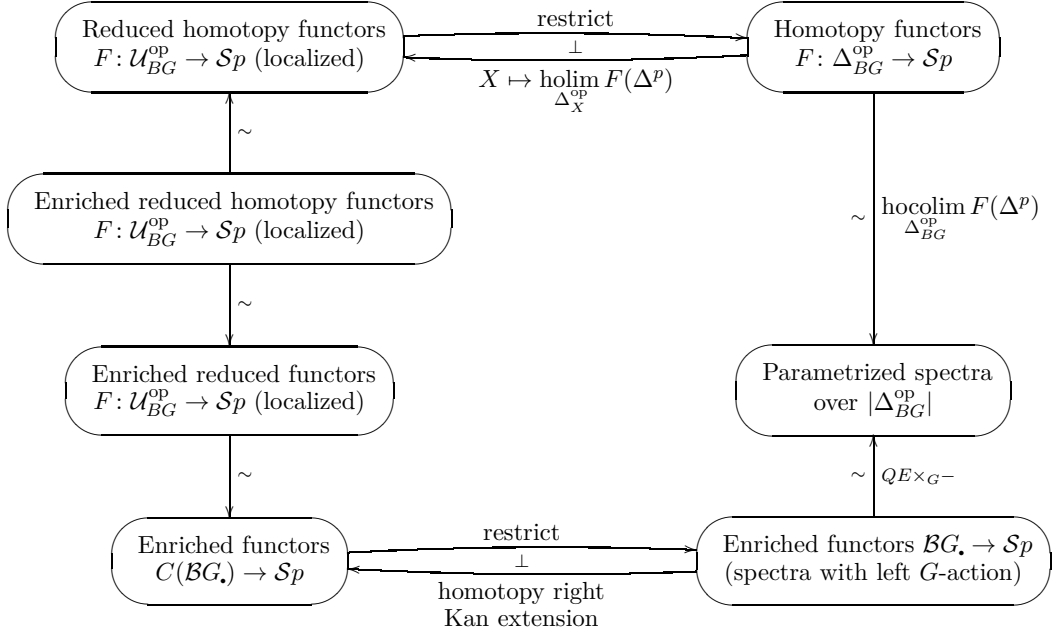
\begin{figure}[h]\label{big diagram of categories}
	\[ \resizebox{\textwidth}{!}{
		\xymatrix @C=1.5in @R=0.5in{
			*+<16pt>[F-:<16pt>]\txt{ Reduced homotopy functors \\ $F\colon \mc U_{BG}^\op \rightarrow \mc Sp$ (localized) } \ar@/^/[r]^-*\txt{restrict} \ar@{}[r]|-{\perp}
			&
			*+<16pt>[F-:<16pt>]\txt{ Homotopy functors \\ $F\colon \Delta_{BG}^\op \rightarrow \mc Sp$ } \ar@/^/[l]^-*\txt{$X \mapsto \underset{\Delta_X^\op}\holim F(\Delta^p)$} \ar[dd]^-*\txt{$\underset{\Delta_{BG}^\op}\hocolim F(\Delta^p)$}_-\sim
			\\
			*+<16pt>[F-:<16pt>]\txt{ Enriched reduced homotopy functors \\ $F\colon \mc U_{BG}^\op \rightarrow \mc Sp$ (localized) }
			\ar[u]_-\sim
			\ar[d]^-\sim
			\\
			*+<16pt>[F-:<16pt>]\txt{ Enriched reduced functors \\ $F\colon \mc U_{BG}^\op \rightarrow \mc Sp$ (localized) }
			\ar[d]^-\sim
			&
			*+<16pt>[F-:<16pt>]\txt{ Parametrized spectra \\ over $|\Delta_{BG}^\op|$ }
			\\
			*+<16pt>[F-:<16pt>]\txt{ Enriched functors \\ $\cG \to \mc Sp$ } \ar@/^/[r]^-*\txt{restrict} \ar@{}[r]|-{\perp} &
			*+<16pt>[F-:<16pt>]\txt{ Enriched functors $\bG \to \mc Sp$ \\ (spectra with left $G$-action) } \ar@/^/[l]^-*\txt{homotopy right \\ Kan extension} \ar [u]_-{QE \times_G -}^-\sim
		}
	} \] 
\end{figure}

To form this natural isomorphism, we assume that $F$ is an enriched reduced homotopy functor on $\mc U_{BG}$. Composing with fibrant replacement, then re-enriching by the equivalences in \autoref{left_hand_equivalences}, we may assume that $F$ sends equivalences of spaces to level equivalences of spectra. We may also compose with $|\Sing -|$ so that it is enriched in topological spaces. These manipulations are natural in $F$, hence we can make these assumptions even if what we are after is an isomorphism that is natural in $F$.

We define
\[ E = \underset{\Delta_{BG}^\op}\hocolim \Map_{BG}(\Delta^p,EG) \]
with $G$ acting on the right on $EG$. By \autoref{contractible} below, $E$ is weakly contractible. Form the following diagram at each spectrum level, in which the second map along the top uses the enriched functoriality of $F$.
\[ \xymatrix{
	QE \times F(EG) \ar[d] \ar[r] &
	\underset{\Delta^p \in \Delta_{BG}^\op}\hocolim \Map_{BG}(\Delta^p,EG) \times F(EG) \ar[d] \ar[r] &
	\underset{\Delta^p \in \Delta_{BG}^\op}\hocolim F(\Delta^p) \\
	QE \times_G F(EG) \ar[r] &
	\underset{\Delta^p \in \Delta_{BG}^\op}\hocolim \Map_{BG}(\Delta^p,EG) \times_G F(EG) \ar@{-->}[ur] &
} \]
This map of spaces induces a map of parametrized spectra over $QE/G \to |\Delta_{BG}^\op|$, or a map from the pushforward of the first to the second over $|\Delta_{BG}^\op|$. To argue that the above map is an equivalence of parametrized spectra, it suffices to argue it is an equivalence at each spectrum level.

To check the composite along the bottom is an equivalence, it suffices to examine the induced map on their homotopy fibers over $|\Delta_{BG}^\op|$. In the target, by a variant of Quillen Theorem B \cite{meyer,quillen2}, the map to $|\Delta_{BG}^\op|$ is a quasifibration, so the fiber $F(\Delta^p)$ is equivalent to the homotopy fiber. In the source, we pick a single $G$-orbit of $QE$ and check that the inclusion of $G \times_G F(EG)$ into the homotopy fiber of $QE \times_G F(EG) \to QE/G$ is an equivalence, by replacing $E$ by a space that is fibrant, then comparing to $EG$. Therefore the above map induces on homotopy fibers a map equivalent to $F(EG) \to F(\Delta^p)$, which is an equivalence because $F$ is a homotopy functor. This proves that the left adjoints commute up to isomorphism.
\end{proof}

\begin{lem}\label{contractible}
	The space $E = \underset{\Delta_{BG}^\op}\hocolim \Map_{BG}(\Delta^p,EG)$ is weakly contractible.
\end{lem}

\begin{proof}
	We first re-arrange the colimit using the following string of weak equivalences.
	\begin{equation}\label{contractible_equivs}
	\xymatrix{
		\underset{\Tw(\Delta_{BG})^\op}\hocolim \Delta^p \times_{BG} EG \ar[r]^-\sim & \underset{\Delta_{BG}}\hocolim \Delta^p \times_{BG} EG \\
		\underset{\Tw(\Delta_{BG})^\op}\hocolim \Delta^p \times \Map_{BG}(\Delta^q,EG) \ar[u]_-\sim \ar[d]^-\sim & \\
		\underset{\Tw(\Delta_{BG})^\op}\hocolim \Map_{BG}(\Delta^q,EG) \ar[r]^\sim & \underset{\Delta_{BG}^\op}\hocolim \Map_{BG}(\Delta^q,EG)
	}
	\end{equation}
	Here $\Tw(\Delta_{BG})^\op$ denotes (the opposite of) the twisted arrow category of $\Delta_{BG}$. The objects are arrows in $\Delta_{BG}$, and a morphism from $\Delta^p \to \Delta^q \to BG$ to $\Delta^{p'} \to \Delta^{q'} \to BG$ is a factorization
	\[ \xymatrix @R=1em{
		\Delta^p \ar[d] \ar[r] &\Delta^{p'} \ar[d] \\
		\Delta^q \ar[d] &\Delta^{q'} \ar[l] \ar[d] \\
		BG \ar@{=}[r] & BG
	}. \]
	In general, for a category $\sC$, the twisted arrow category $\Tw(\sC)^\op$ is equipped with a ``source'' functor $s\colon \Tw(\sC)^\op \to C$ that remembers just the source of each arrow, and a ``target'' functor $t\colon \Tw(\sC)^\op \to \sC^\op$ that remembers  the target of the arrow.
	
	It is straightforward to define the diagrams on the left-hand side of \autoref{contractible_equivs}. The top horizontal map is the pullback of a diagram on $\Delta_{BG}$ along the source functor. Similarly, the horizontal diagram on the bottom is a pullback along the target functor.  The bottom vertical arrow arises by collapsing $\Delta^p$ to a point and is thus a levelwise equivalence. The top vertical arrow arises from the levelwise maps 
	\[ \xymatrix{ \Delta^p \times \Map_{BG}(\Delta^q,EG) \ar[r] & \Delta^p \times_{BG} EG } \]
	defined by sending $(x,f) \mapsto (x,f(g(x)))$, where $g$ is the given map $\Delta^p \to \Delta^q$. We check from the definition that this is indeed a map of $\Tw(\Delta_{BG})^\op$-diagrams. It is also an equivalence on each term, since restricting the $\Delta^p$ or $\Delta^q$ to a single point is an equivalence, and after this substitution we get a homeomorphism
	\[ \xymatrix{ \Map_{BG}(\{*\},EG) \ar[r]^-\cong & \{*\} \times_{BG} EG. } \]

	The next step is to show that these four maps of colimits are weak equivalences. For the vertical maps, this follows because the two maps of diagrams are an equivalence on each term. For the horizontal arrows, this follows because the source and target functors are homotopy terminal. For the source functor, this means that for any object $j \in C$, the overcategory $(j \downarrow s)$ is contractible. To prove this, we note that the overcategory consists of pairs of arrows $j \to a \to b$ and morphisms of the form
	\[ \xymatrix{
		j \ar[d] \ar@{=}[r] & j \ar[d] \\
		a \ar[d] \ar[r] & c \ar[d] \\
		b & d \ar[l]
	}. \]
	The inclusion of the subcategory of all arrows of the form $j = j \to b$ has a right adjoint, so that subcategory has an equivalent nerve. Furthermore, this subcategory has a terminal object $j = j = j$, so it is contractible. All together, this proves that $s$ is homotopy terminal. A similar proof works for the target functor $t$. 
	
	We have now reduced to proving that $\underset{\Delta_{BG}}\hocolim (\Delta^p \times_{BG} EG)$ is weakly contractible. Since geometric realization commutes with finite limits, we get a homeomorphism
	\[ \underset{\Delta_{BG}}\hocolim \left(\Delta^p \times_{BG} EG\right) \cong \left(\underset{\Delta_{BG}}\hocolim \Delta^p\right) \times_{BG} EG. \]
	Clearly $BG \times_{BG} EG \cong EG$ is contractible, so it is enough to prove that the map
	\[ \phi\colon \underset{\Delta_{BG}}\hocolim \Delta^p \to BG, \]
	which arises from all the individual maps $\Delta^p \to BG$, is an equivalence. There is an immediate equivalence
	\begin{equation}\label{easy_equivalence}
		\xymatrix{ \underset{\Delta_{BG}}\hocolim \Delta^p \ar[r]^-\sim & \underset{\Delta_{BG}}\hocolim {*} \ar[r]^-\cong & |\Delta_{BG}| \ar[r]^-\sim & BG }
	\end{equation}
	but that is a different map.
	To show that $\phi$ is an equivalence, we extend it to a natural transformation of functors on unbased spaces
	\[ \xymatrix{ \underset{\Delta_{X}}\hocolim \Delta^p \ar[r] & X. } \]
	It is clearly an equivalence when $X$ is empty or contractible. Furthermore, using \eqref{easy_equivalence}, both sides are equivalent to the identity functor and are therefore excisive. A standard inductive argument then shows that $\phi$ is an equivalence on all spaces. This finishes the proof.
\end{proof}

\section{Review of coarse and bivariant $A$-theory}

Let $G$ be a finite group and $X$ a $G$-space. Let $R(X)$ be the category of retractive spaces
\[ X\xrightarrow{i} Y\xrightarrow{r} X, \qquad ri=\id, \]
with weak equivalences given by the weak homotopy equivalences and cofibrations given by maps that have the fiberwise homotopy extension property (FHEP). The category $R(X)$ has a $G$-action through exact functors induced by conjugation from the $G$-action on $X$ \cite[\S 3.1.]{CaryMona}.   For taking $K$-theory, we restrict to the subcategory $R_{hf}(X)  \subseteq R(X)$ of retractive spaces that are \emph{homotopy finite}. These are the spaces that, in the homotopy category of retractive spaces, are a retract of a finite cell complex relative to $X$. We note the action respects this condition. 

For each subgroup $H \leq G$, the homotopy fixed points are defined as
$$R_{hf}(X)^{hH}:=\Cat(\tG, R_{hf}(X))^H ,$$ where $\tG$ is the $G$-category  with one object for each element of $G$ and a unique morphism between any two objects, and $\Cat(\tG, R_{hf}(X))$ is the category of all functors and natural transformations, with $G$ acting by conjugation \cite[Definition 2.2.]{CaryMona}. 

The homotopy fixed point category $R_{hf}(X)^{hH}$ is equivalent to the Waldhausen category whose objects are $H$-spaces $Y$ containing $X$ as an $H$-equivariant retract, whose underlying space is homotopy finite \cite[Proposition 3.1.]{CaryMona}. The morphisms are the $H$-equivariant maps of retractive spaces $Y \to Y'$. The cofibrations are the $H$-equivariant maps which are nonequivariantly cofibrations and the weak equivalences are the $H$-equivariant maps which are nonequivariantly weak equivalences.

We define $\bA_G^{\! \textup{coarse}}(X)$ to be the na\"ive $G$-spectrum obtained by applying $S_{\sbt}$ to the Waldhausen $G$-category $\Cat(\tG,R_{hf}(X))$. This is equivalent to the underlying na\"ive $G$-spectrum of a genuine $\Omega$-$G$-spectrum \cite[Theorem 2.21.]{CaryMona}.

For a Hurewicz fibration $p\colon E \to B$, the bivariant $A$-theory $A(p)$ is defined to be the $K$-theory of the Waldhausen category of retractive spaces $X$ over $E$, with the property that $X \to B$ is a fibration, and the map of fibers $E_b \to X_b$ is a retract up to homotopy of a relative finite complex. See \cite{bruce,raptis_steimle}.

In the present section we extend the following result of \cite{CaryMona} to the coassembly map.
\begin{prop}\label{coarse_equals_bivariant}
	There is a natural equivalence of symmetric spectra
	\[ \bA_G^{\textup{coarse}}(X)^H \simeq \bA(EG \times_H X \to BH). \]
\end{prop}
The equivalence is induced by the functor
\[ \Phi\colon R_{hf}(X)^{hH} \ra R_{hf}(EG \times_H X \overset{p}\to BH) \]
that applies $EG \times_H -$ to the retractive space $(Y,i_Y,p_Y)$ over $X$, obtaining a retractive space over $EG \times_H X$:
\[ \xymatrix @C=5em{ EG \times_H X \ar[r]^-{EG \times_H i_Y} & EG \times_H Y \ar[r]^-{EG \times_H p_Y} & EG \times_H X. } \]

To define the coassembly map, we observe that while bivariant $A$-theory is a functor of fibrations, it can be regarded as a contravariant functor on $\mc U_B$ in the following way. Fix a fibration $p\colon E \to B$. Then $\mc U_B$ is equivalent to the category whose objects are pullback squares
\[ \xymatrix{ E' \ar[r] \ar[d]_-{p'} & E \ar[d]^-p \\ B' \ar[r] & B } \]
and whose maps are commuting squares (necessarily pullback squares)
\[ \xymatrix{ E'' \ar[r] \ar[d]_-{p''} & E' \ar[d]^-{p'} \\ B'' \ar[r] & B' .} \]
Along this equivalence, bivariant $A$-theory is a reduced homotopy functor from $\mc U_B^\op$ to spectra, so it has a coassembly map
\[ c\alpha \colon \bA(E' \overset{p'}\to B') \to \bA_{\%}(E' \overset{p'}\to B'). \]
We emphasize that the coassembly map depends on the choice of fibration $E\overset{p}\to B$ and map $B'\to B$. Different choices give rise to different coassembly maps.

Fix the fibration $EG \times_H X \to BH$ and the pullback square 
\[ \xymatrix{ EG \times_H X \ar[r]^-= \ar[d]_-{p} & EG \times_H X \ar[d]^-p \\ BH \ar[r]^-= & BH, } \]
and consider the resulting coassembly map. Our last remaining goal is to prove the following.

\begin{thm}\label{thm:homotopy_fixed_equals_coassembly}
	In the stable homotopy category, the map from fixed points to homotopy fixed points is isomorphic to the coassembly map for bivariant $A$-theory:
	\[ \xymatrix{
		\bA_G^{\textup{coarse}}(X)^H \ar[d]^-\sim \ar[r] & \bA_G^{\textup{coarse}}(X)^{hH} \ar[d]^-\sim \\
		\bA(EG \times_H X \to BH) \ar[r]^-{c\alpha} & \bA_{\%}(EG \times_H X \to BH). } \]
	Furthermore the left-hand map in the above diagram can be taken to be the equivalence of \autoref{coarse_equals_bivariant}.
\end{thm}

\section{Proof of  \autoref{intro_second}}

Note that without loss of generality we may take $H = G$. Since $G$ is finite, we may ignore issues of enrichment. By \autoref{coassemblymap}, the coassembly map for bivariant $A$-theory is equivalent to the homotopy limit map for the diagram on $C(\mathcal{B} G)$ given by bivariant $A$-theory on $EG$ and $BG$. So it remains to compare the resulting diagram on $C(\mathcal{B} G)$ to the one defined by coarse $A$-theory.

%
%
\begin{prop}
The equivalence of \autoref{coarse_equals_bivariant} can be extended to an equivalence of diagrams of symmetric spectra over $C(\mathcal{B} G)$.
\end{prop}
We expect it is possible to compare these two as diagrams over $\mc O(G)^\op$, but this raises additional coherence issues, and is not necessary to prove \autoref{thm:homotopy_fixed_equals_coassembly}. 
\begin{proof}
We start by describing the $\mc O(G)^\op$-action on bivariant $A$-theory. To each map of $G$-sets $f\colon G/H \to G/K$ and $G$-space $X$ we assign the following pullback square.
\[ \xymatrix{
B(*,G,G \times_H X) \ar[r] \ar[d] & B(*,G,G \times_K X) \ar[d] \\
B(*,G,G/H) \ar[r]^-{EG \times_G f} & B(*,G,G/K) } \]
The vertical maps collapse $X$ to a point, and the top horizontal map $$G \times_H X \to G \times_KX$$ sends $(\gamma,x)$ to $(\gamma g^{-1},gx)$, where $g$ is any element such that $f(eH)=g^{-1}K$. Note that this formula is well-defined because $g$ is unique up to left multiplication by $K$. It is easy to check that these formulas give a functor from $\mc O(G)$ into the category of pullbacks of the fibration $EG \times_G X \to BG$, and therefore define the action of $\mc O(G)^\op$ on the bivariant $A$-theory spectra $\bA(EG \times_H X \to EG/H)$.
This action is strict by functoriality of bivariant $A$-theory (see \cite[Rmk 3.5]{raptis_steimle}).

Now we restrict to $C(\mathcal{B} G)$, where we wish to prove that the functor $\Phi$ of \autoref{coarse_equals_bivariant} gives a map of $C(\mathcal{B} G)$ diagrams, in other words that the two squares below commute:
\[ \xymatrix{
\bA_G^{\textup{coarse}}(X)^G \ar[r]^-\Phi_-\sim \ar[d]_-{\textup{include}} & \bA(EG \times_G X \to EG/G) \ar[d]^-{\textup{include}} \\
\bA_G^{\textup{coarse}}(X)^{\{e\}} \ar[r]^-\Phi_-\sim \ar[d]_-{g\cdot} & \bA(EG \times X \to EG) \ar[d]^-{g\cdot} \\
\bA_G^{\textup{coarse}}(X)^{\{e\}} \ar[r]^-\Phi_-\sim & \bA(EG \times X \to EG).
} \]
This turns out to be false, but only because the relevant functors of Waldhausen categories agree up to canonical isomorphism, rather than strictly. We therefore replace our two diagrams over $C(\mathcal{B} G)$ by equivalent ones on which the map $\Phi$ strictly commutes with the $C(\mathcal{B} G)$ action.

First we make the following reduction. We first show that in order to get a strictly commuting zig-zag of equivalences of  $C(\mathcal{B} G)$-diagrams, it is enough to define a square of $G$-equivariant functors
\[ \xymatrix{
	\mc C \ar[r]^-{F_1} \ar[d]_-I & \mc C' \ar[d]^-{I'} \\
	\mc D \ar[r]_-{F_2} & \mc D'
} \]
such that $\mc C$ and $\mc C'$ have trivial $G$-action, and such that the square commutes up to a $G$-fixed natural isomorphism $\eta$. Given such a square, we may replace $\mc D$ by the category $\mc D_I$ defined as follows:
\begin{itemize}
\item the objects $\mc D_I$ are $\ob \mc C \amalg \ob \mc D$, and 
\item the morphisms are given by $\mc D_I(d,d')=\mc D(d,d')$, $\mc D_I(d,c)=\mc D(d,Ic)$, $\mc D_I(c,d)=\mc D(Ic,d)$ if $c$ is an object of $\mc C$ and $d,d'$ are objects of $\mc D$.
\end{itemize}
We define a new functor $\mc D_I \to \mc D'$ using $F_2$ on the full subcategory on $\ob \mc D$, $I' \circ F_1$ on the full subcategory on $\ob \mc C$, and on each morphism $f$ between $c \in \ob \mc C$ and $d \in \ob \mc D$, the composite
\[ \xymatrix @C=3em{
	I' \circ F_1(c) \ar@{<->}[r]^-\cong_-\eta & F_2 \circ I(c) \ar@{<->}[r]_-{F_2(f)} & F_2(d).
} \]
It is easy to check this is indeed a functor and is $G$-equivariant. It is then straightforward to define the rest of the following diagram so that every functor is equivariant and every square of functors commutes strictly, giving a zig-zag of $C(\mathcal{B} G)$-diagrams of categories
\[ \xymatrix{
	\mc C \ar@{=}[r] \ar[d]_-I & \mc C \ar[r]^-{F_1} \ar[d] & \mc C' \ar[d]^-{I'} \\
	\mc D & \ar[l]^-\sim \mc D_I \ar[r] & \mc D' .
} \]
Note that if $\mc C$ and $\mc D$ are Waldhausen categories and all functors $I,I', F_1, F_2$ are exact, then the resulting diagram above is also a diagram of Waldhausen categories, where $\mc D_I$ has the Waldhausen structure inherited from computing maps in $\mc D$.  With this reduction in hand, it is enough to make a square of functors of Waldhausen $G$-categories, in which the top row has trivial $G$-action, that commutes up to a $G$-fixed natural isomorphism. We will construct the following square:
\[ \xymatrix @C=4em{
	\Cat(\tG,R_{hf}(X))^G \ar[r]^-\Phi \ar[dd]_-{I} & R_{hf}(EG \times_G X \to BG) \ar[d]^-{q^*} \\
	& R_{hf}(EG \times X \to EG) \ar[d]^-{const} \\
	\Cat(\tG,R_{hf}(X)) \ar[r]^-{\widetilde{\Phi}} & \Cat(\tG,R_{hf}(EG \times X \to EG)).
} \]
The map $\Phi$ along the top is the one from \autoref{coarse_equals_bivariant} that applies $EG \times_H -$ to the retractive space $(Y,i_Y,p_Y)$ over $X$, obtaining a retractive space over $EG \times_H X$.

The left-hand vertical map $I$ includes the fixed points into the whole category, i.e. it takes a retractive $G$-space $(Y,i,p)$ to the $G$-tuple of retractive spaces $(Y,i\circ g^{-1},g \circ p)$ with isomorphisms of retractive spaces
\[ \xymatrix @C=4em{ \phi_{g,h}\colon (Y,i\circ g^{-1},g \circ p) \ar[r]^-{h^{-1}g \cdot -} & (Y,i\circ h^{-1},h \circ p) } \]
over the identity map of $X$. Along the right-hand edge, the first functor pulls back along the quotient map
\[ q\colon EG \times X \to EG \times_G X \]
The left action of $g \in G$ on the target is by pullback along the map
\[ \xymatrix @C=4em{ \rho_g\colon EG \times X \ar[r]^-{- \cdot g \times g^{-1} \cdot -} & EG \times X } \]
and note that $q^*$ lands in the $G$-fixed points because the composite function $q \circ \rho_g$ is equal to $q$. The second functor on the right-hand edge pulls back along the map of categories $\tG \to *$. To define the functor on the bottom, first form the functor
\[ \Phi\colon R_{hf}(X) \to R_{hf}(EG \times X \to EG), \]
\[ \Phi(Z,i,p) = EG \times (Z,i,p) = (EG \times Z,\id \times i, \id \times p). \]
Then pick the isomorphisms
\[ \theta_g\colon \Phi \circ g \to g \circ \Phi \]
\[ EG \times (Z,i \circ g^{-1},g \circ p) \to \rho_g^*(EG \times (Z,i,p)) \]
arising from the commuting diagram
\[ \xymatrix @C=3em{
	EG \times X \ar[r]^-{\cdot g,g^{-1}\cdot}_-{\rho_g} \ar[d]_-{\id,i \circ g^{-1}} & EG \times X \ar[d]_-{\id,i} \\
	EG \times Z \ar[r]^-{\cdot g,\id} \ar[d]_-{\id,g \circ p} & EG \times Z \ar[d]_-{\id,p} \\
	EG \times X \ar[r]^-{\cdot g,g^{-1}\cdot}_-{\rho_g} & EG \times X. \\
} \]
We check the cocycle condition $g\theta_h \circ \theta_g = \theta_{gh}$, which reduces to the equality $(- \cdot g) \cdot h = - \cdot (gh)$ as self-maps of $EG \times Z$, and $\rho_h \circ \rho_g = \rho_{gh}$ as self-maps of $EG \times X$. Therefore by \cite[Def 2.5]{CaryMona}, the isomorphisms $\theta_g$ make $\Phi$ a pseudoequivariant functor. By \cite[Proposition 2.10.]{CaryMona}, after applying $\Cat(\tG,-)$ we get a strictly equivariant functor $\widetilde\Phi$.

The top route through our diagram of functors takes a retractive $G$-space $Y$ over $X$ to the functor $\tG \to R_{hf}(EG \times X \to EG)$ with values
\[ g \mapsto q^*(EG \times_G (Y,i,p)), \quad (g \to h) \mapsto \id \]
The bottom route produces the functor with values
\[ g \mapsto \rho_g^*(EG \times (Y,i,p)). \]
To describe the maps, let us represent the space $\rho_g^*(EG \times (Y,i,p))$ by drawing the span along which we take the pullback to get it:
\[ \xymatrix{ EG \times Y \ar[r]^-{\id,p} & EG \times X & \ar[l]_-{\cdot g, g^{-1} \cdot}^-{\rho_g} EG \times X & \rho_g^*(EG \times (Y,i,p)) } \]
Then our functor out of $\tG$ assigns the map $g \to h$ to the composite of the following isomorphisms.
\begin{equation}\tag{$\ast$}\label{coassembly_proof_big_composite}
\xymatrix @C=4em{
	EG \times Y \ar[r]^-{\id,p} \ar[d]_-{\cdot g^{-1},\id} & EG \times X \ar[d]^-{\cdot g^{-1},g\cdot } & \ar[l]_-{\cdot g, g^{-1} \cdot}^-{\rho_g} EG \times X \ar@{=}[d] & \rho_g^*(EG \times (Y,i,p)) \ar[d]^-{\theta_g^{-1}} \\
	EG \times Y \ar[r]^-{\id,g \circ p} \ar[d]_-{\id, h^{-1}g \cdot} & EG \times X \ar@{=}[d] & \ar@{=}[l] EG \times X \ar@{=}[d] & EG \times (Y,i \circ g^{-1},g \circ p) \ar[d]^-{\id, (h^{-1}g \cdot)} \\
	EG \times Y \ar[r]^-{\id,h \circ p} \ar[d]_-{\cdot h,\id} & EG \times X \ar[d]^-{\cdot h, h^{-1}\cdot } & \ar@{=}[l] EG \times X \ar@{=}[d] & EG \times (Y,i \circ h^{-1},h \circ p) \ar[d]^-{\theta_h} \\
	EG \times Y \ar[r]^-{\id,p} & EG \times X & \ar[l]_-{\cdot h, h^{-1} \cdot}^-{\rho_h} EG \times X & \rho_h^*(EG \times (Y,i,p)) \\
}
\end{equation} 
Now we will define a natural isomorphism $\eta$ from the bottom route to the top route. Continuing to use this span notation, for each $g \in \tG$ we define an isomorphism $\eta_g$ by the map of spans
\[ \xymatrix{
	EG \times Y \ar[r]^-{\id,p} \ar[d]_-{\id,\id} & EG \times X \ar[d]^-q & EG \times X \ar[l]_-{\rho_g} \ar@{=}[d] & \rho_g^*(EG \times (Y,i,p)) \ar[d]^-{\eta_g} \\
	EG \times_G Y \ar[r]^-{\id,p} & EG \times_G X & EG \times X \ar[l]_-q & q^*(EG \times_G (Y,i,p))
} \]
This commutes with the maps $g \to h$ of $\tG$ because the composite of the three maps of spans from \eqref{coassembly_proof_big_composite} commutes with the map of spans just above. Naturality follows because each $G$-equivariant map $Y \to Y'$ induces maps on the source and target of $\eta_g$ that commute with $\eta_g$ for each $g$. Finally we check that $\eta$ is a $G$-fixed natural transformation. The map $\gamma\eta_{\gamma^{-1}g} := \rho_\gamma^*\eta_{\gamma^{-1}g}$ comes from the map of spans
\[ \xymatrix{
	EG \times Y \ar[r]^-{\id,p} \ar[d]_-{\id,\id} & EG \times X \ar[d]^-q & EG \times X \ar[l]^-{\rho_{\gamma^{-1}g}} \ar@{=}[d] & EG \times X \ar[l]^-{\rho_\gamma} \ar@/_1em/[ll]_-{\rho_g} \ar@{=}[d] & \rho_\gamma^*\rho_{\gamma^{-1}g}^*(EG \times (Y,i,p)) \ar[d]^-{\rho_\gamma^* \eta_{\gamma^{-1}g}} \\
	EG \times_G Y \ar[r]^-{\id,p} & EG \times_G X & EG \times X \ar[l]_-q & EG \times X \ar[l]_-{\rho_\gamma}  \ar@/^1em/[ll]^-{q} & \rho_\gamma^*q^*(EG \times_G (Y,i,p)),
} \]
which is indeed the same map of spans that defines $\eta_g$. This finishes the construction of the square of equivariant functors that commutes up to equivariant isomorphism. In summary, using the reduction cited earlier in the proof, we have now constructed a strictly commuting zig-zag of $C(\mathcal{B} G)$-diagrams of Waldhausen categories
\[ \small \xymatrix @C=1em{
	\Cat(\tG,R_{hf}(X))^G \ar[d] \ar@{=}[r] & \Cat(\tG,R_{hf}(X))^G \ar[r]^-\Phi \ar[d] & R_{hf}(EG \times_G X \to BG)  \ar@{=}[r] \ar[d]^-{\mathrm{const} \circ q^*} & R_{hf}(EG \times_G X \to BG) \ar[d]^-{q^*} \\
	\Cat(\tG,R_{hf}(X)) & \ar[l]_-\sim \Cat(\tG,R_{hf}(X))_I \ar[r] & \Cat(\tG,R_{hf}(EG \times X \to EG)) & R_{hf}(EG \times X \to EG). \ar[l]_-{\mathrm{const}}^-\sim
} \]
Now we apply the $K$-theory functor to this diagram. By \autoref{coarse_equals_bivariant}, the top map $\Phi$ induces an equivalence in $K$-theory. The bottom maps labeled $\sim$ are $G$-maps which are nonequivariant equivalences. It remains to show that the remaining horizontal map gives an equivalence on $K$-theory. In general, for any pseudo equivariant functor $\Phi\colon \mc C\to \mc D$, we have a commutative diagram of nonequivariant categories
	\[\xymatrix{ \Cat(\tG, \mc C) \ar[d]_-\sim \ar[r]^-{\widetilde{\Phi}} & \Cat(\tG, \mc D)\ar[d]^-\sim \\ 
		\mc C \ar[r]_-\Phi & \mc D }\]
where the vertical maps are nonequivariant equivalences. (Note that the diagram with those equivalences reversed does not commute.) Since $\Phi$ induces an equivalence on $K$-theory, so does $\widetilde{\Phi}$. Now use the factorization
\[ \xymatrix{ 	\Cat(\tG,R_{hf}(X))\ar@/_2em/[rr]_-{\widetilde{\Phi}} \ar[r]^-\sim & \Cat(\tG,R_{hf}(X))_I \ar[r] & \Cat(\tG,R_{hf}(EG \times X \to EG))  }\]
to conclude that the remaining functor $$\Cat(\tG,R_{hf}(X))_I \to \Cat(\tG,R_{hf}(EG \times X \to EG))$$ also gives an equivalence in $K$-theory. Thus we get a strictly commuting zig zag of equivalences of $C(\mathcal{B} G)$ diagrams in spectra.

\end{proof}

  \bibliographystyle{amsalpha}
  \bibliography{references}

\providecommand{\bysame}{\leavevmode\hbox to3em{\hrulefill}\thinspace}
\providecommand{\MR}{\relax\ifhmode\unskip\space\fi MR }
\providecommand{\MRhref}[2]{%
  \href{http://www.ams.org/mathscinet-getitem?mr=#1}{#2}
}
\providecommand{\href}[2]{#2}
\begin{thebibliography}{MMSS01}

\bibitem[ABG11]{ando_blumberg_gepner}
M.~Ando, A.J. Blumberg, and D.~Gepner, \emph{Parametrized spectra,
  multiplicative {T}hom spectra, and the twisted {U}mkehr map}, arXiv preprint
  arXiv:1112.2203 (2011).

\bibitem[BGS]{Gmonster2}
Clark Barwick, Saul Glasman, and Jay Shah, \emph{Spectral {M}ackey functors and
  equivariant algebraic {$K$}-theory {(II)}}, arXiv:1505.03098v2.

\bibitem[BK72]{bousfield_kan}
A.~K. Bousfield and D.~M. Kan, \emph{Homotopy limits, completions and
  localizations}, Lecture Notes in Mathematics, Vol. 304, Springer-Verlag,
  Berlin-New York, 1972. \MR{0365573}

\bibitem[CK09]{cohen_umkehr}
Ralph~L. Cohen and John~R. Klein, \emph{Umkehr maps}, Homology Homotopy Appl.
  \textbf{11} (2009), no.~1, 17--33. \MR{2475820}

\bibitem[DH04]{devinatzhopkins}
Ethan~S. Devinatz and Michael~J. Hopkins, \emph{Homotopy fixed point spectra
  for closed subgroups of the morava stabilizer groups}, Topology \textbf{43}
  (2004), no.~1, 1 -- 47.

\bibitem[DL98]{DavisLuck}
James~F. Davis and Wolfgang L\"{u}ck, \emph{Spaces over a category and assembly
  maps in isomorphism conjectures in {$K$}- and {$L$}-theory}, $K$-Theory
  \textbf{15} (1998), no.~3, 201--252. \MR{1659969}

\bibitem[DWW03]{dww}
Wiliam Dwyer, Michael Weiss, and Bruce Williams, \emph{A parametrized index
  theorem for the algebraic {$K$}-theory {E}uler class}, Acta mathematica
  \textbf{190} (2003), no.~1, 1--104.

\bibitem[FJ93]{farrell_jones}
F.~T. Farrell and L.~E. Jones, \emph{Isomorphism conjectures in algebraic
  {$K$}-theory}, J. Amer. Math. Soc. \textbf{6} (1993), no.~2, 249--297.
  \MR{1179537}

\bibitem[GJ09]{goerss_jardine}
Paul~G Goerss and John~F Jardine, \emph{Simplicial homotopy theory}, Springer
  Science \& Business Media, 2009.

\bibitem[Gra76]{quillen2}
Daniel Grayson, \emph{Higher algebraic {$K$}-theory. {II} (after {D}aniel
  {Q}uillen)}, Algebraic {$K$}-theory ({P}roc. {C}onf., {N}orthwestern {U}niv.,
  {E}vanston, {I}ll., 1976), Springer, Berlin, 1976, pp.~217--240. Lecture
  Notes in Math., Vol. 551. \MR{0574096 (58 \#28137)}

\bibitem[GW99]{goodwillie_weiss2}
Thomas~G. Goodwillie and Michael Weiss, \emph{Embeddings from the point of view
  of immersion theory. {II}}, Geom. Topol. \textbf{3} (1999), 103--118.
  \MR{1694808}

\bibitem[Hea17]{drewhomotopylimit}
Drew Heard, \emph{The homotopy limit problem and the cellular picard group of
  hermitian k-theory}, https://arxiv.org/abs/1705.02810 (2017).

\bibitem[HKO11]{hukrizormsby}
P.~Hu, I.~Kriz, and K.~Ormsby, \emph{The homotopy limit problem for hermitian
  k-theory, equivariant motivic homotopy theory and motivic real cobordism},
  Advances in Mathematics \textbf{228} (2011), no.~1, 434 -- 480.

\bibitem[HP04]{HP}
Ian Hambleton and Erik~Kjaer Pedersen, \emph{Identifying assembly maps in k-
  and l-theory}, Mathematische Annalen \textbf{328} (2004), 27--57.

\bibitem[Lew78]{lewis_thesis}
L.~G. Lewis, \emph{The stable category and generalized {T}hom spectra}, Ph.D.
  thesis, University of Chicago, Department of Mathematics, 1978.

\bibitem[Lil73]{Lillig}
Joachim Lillig, \emph{A union theorem for cofibrations}, Arch. Math. (Basel)
  \textbf{24} (1973), 410--415. \MR{0334193}

\bibitem[LM]{lind_malkiewich_morita}
John Lind and Cary Malkiewich, \emph{The {M}orita equivalence between
  parametrized spectra and module spectra}, New Directions in Homotopy Theory
  (Contemporary Mathematics) (to appear).

\bibitem[L{\"u}c19]{luckassembly}
Wolfgang L{\"u}ck, \emph{Assembly maps}, https://arxiv.org/pdf/1805.00226.pdf
  (2019).

\bibitem[Mal15]{malkiewich2015tower}
Cary Malkiewich, \emph{A tower connecting gauge groups to string topology},
  Journal of Topology \textbf{8} (2015), no.~2, 529--570.

\bibitem[Mal17]{coassembly}
\bysame, \emph{Coassembly and the {$K$}-theory of finite groups}, Advances in
  Mathematics \textbf{307} (2017), 100--146.

\bibitem[Mey86]{meyer}
Jean-Pierre Meyer, \emph{Bar and cobar constructions. {II}}, J. Pure Appl.
  Algebra \textbf{43} (1986), no.~2, 179--210. \MR{866618}

\bibitem[MM]{CaryMona}
Cary Malkiewich and Mona Merling, \emph{Equivariant {$A$}-theory}, available as
  arXiv:1609.03429.

\bibitem[MMSS01]{mmss}
M.~A. Mandell, J.~P. May, S.~Schwede, and B.~Shipley, \emph{Model categories of
  diagram spectra}, Proc. London Math. Soc. (3) \textbf{82} (2001), no.~2,
  441--512. \MR{1806878 (2001k:55025)}

\bibitem[MS06]{ms}
J.~P. May and J.~Sigurdsson, \emph{Parametrized homotopy theory}, Mathematical
  Surveys and Monographs, vol. 132, American Mathematical Society, Providence,
  RI, 2006. \MR{2271789}

\bibitem[OR17]{rondingsetc}
Paul Arne~Ostv{\ae}r Oliver~R{\"o}ndigs, Markus~Spitzweck, \emph{The motivic
  hopf map solves the homotopy limit problem for $k$-theory},
  https://arxiv.org/pdf/1701.06144.pdf (2017).

\bibitem[Rie14]{riehl_cat_htpy}
Emily Riehl, \emph{Categorical homotopy theory}, no.~24, Cambridge University
  Press, 2014.

\bibitem[RS14]{raptis_steimle}
George Raptis and Wolfgang Steimle, \emph{On the map of {B\"o}kstedt--{M}adsen
  from the cobordism category to {$A$}--theory}, Algebraic \& Geometric
  Topology \textbf{14} (2014), no.~1, 299--347.

\bibitem[Str]{strickland}
N.P. Strickland, \emph{The category of {CGWH} spaces},
  https://neil-strickland.staff.shef.ac.uk/courses/homotopy/cgwh.pdf.

\bibitem[Tho83]{homotopylimit}
R.~W. Thomason, \emph{The homotopy limit problem}, Proceedings of the
  {N}orthwestern {H}omotopy {T}heory {C}onference ({E}vanston, {I}ll., 1982)
  (Providence, R.I.), Contemp. Math., vol.~19, Amer. Math. Soc., 1983,
  pp.~407--419. \MR{711065 (84j:18012)}

\bibitem[Wal85]{waldhausen}
Friedhelm Waldhausen, \emph{Algebraic {$K$}-theory of spaces}, Algebraic and
  geometric topology ({N}ew {B}runswick, {N}.{J}., 1983), Lecture Notes in
  Math., vol. 1126, Springer, Berlin, 1985, pp.~318--419. \MR{802796
  (86m:18011)}

\bibitem[Wei99]{goodwillie_weiss1}
Michael Weiss, \emph{Embeddings from the point of view of immersion theory.
  {I}}, Geom. Topol. \textbf{3} (1999), 67--101. \MR{1694812}

\bibitem[Wil00]{bruce}
Bruce Williams, \emph{Bivariant {R}iemann {R}och theorems}, Geometry and
  topology: {A}arhus (1998), Contemp. Math., vol. 258, Amer. Math. Soc.,
  Providence, RI, 2000, pp.~377--393. \MR{1778119 (2002c:57052)}

\bibitem[WW95a]{weisswilliamsassembly}
Michael Weiss and Bruce Williams, \emph{Assembly}, Novikov conjectures, index
  theorems and rigidity, {V}ol. 2 ({O}berwolfach, 1993), London Math. Soc.
  Lecture Note Ser., vol. 227, Cambridge Univ. Press, Cambridge, 1995,
  pp.~332--352. \MR{1388318}

\bibitem[WW95b]{weiss_williams_assembly}
\bysame, \emph{Assembly}, Novikov conjectures, index theorems and rigidity,
  {V}ol.\ 2 ({O}berwolfach, 1993), London Math. Soc. Lecture Note Ser., vol.
  227, Cambridge Univ. Press, Cambridge, 1995, pp.~332--352. \MR{1388318}

\end{thebibliography}

\begingroup%
\setlength{\parskip}{\storeparskip}

\end{document}